\documentclass[11pt, a4paper]{article}

\usepackage{amsmath,amsfonts,amssymb}
\usepackage{bbm}
\usepackage{cases}

\usepackage{kotex}
\usepackage{lipsum}
\usepackage{subcaption}

\usepackage{xcolor}
\usepackage{graphicx}
\usepackage{ulem}

\newtheorem{defn}{Definition}[section]
\newtheorem{thm}[defn]{Theorem}
\newtheorem{lem}[defn]{Lemma}

\newtheorem{cor}[defn]{Corollary}
\newtheorem{rmk}[defn]{Remark}

\newcommand\blfootnote[1]{%
  \begingroup
  \renewcommand\thefootnote{}\footnote{#1}%
  \addtocounter{footnote}{-1}%
  \endgroup
}

\newenvironment{proof}{{\bf Proof }}{{\vskip 0.1cm \hfill$\Box$}}

\begin{document} 

\noindent
{\Large \bf Notes on linear elliptic equations with $L^2$-gradient perturbations and singular zero-order coefficients}
\\ \\
\bigskip
\noindent
{\bf Haesung Lee}  \\
\noindent
\blfootnote{This research was supported by Kumoh National Institute of Technology (2025\,$\thicksim$\,2026).}
\noindent
{\bf Abstract.}  
In this paper, we study the existence, uniqueness, and quantitative estimates for weak solutions to linear elliptic Dirichlet problems of the form
\[
-\operatorname{div}(\gamma \nabla u)+\langle \nabla\phi+\mathbf{H},\nabla u\rangle+(c+\alpha)u=f \quad\text{ in }U,
\quad\;
u=0  \quad\text{on }\partial U,
\]
where $U\subset \mathbb{R}^d$ is bounded, $\gamma\in[1,\infty)$ is a constant, $\phi\in H^{1,2}(U)\cap L^\infty(U)$, $\mathbf{H}\in L^p(U,\mathbb{R}^d)$ for some $p \in (d, \infty)$, and $c\in L^1(U)$ with $c\ge0$. A key feature of this setting is that the drift contains the low-regularity term $\nabla\phi$, which is only assumed to belong to $L^2(U,\mathbb{R}^d)$, while the zero-order coefficient is merely integrable. We prove that, even under these rough assumptions, well-posedness and quantitative energy and $L^2$ estimates remain valid. In addition, by using a previously established interpolation result, we characterize a trade-off between the integrability of the source term $f$ and that of the zero-order coefficient $c$, and show that well-posedness together with the corresponding quantitative estimates hold under these interpolated assumptions.
\\ \\
\noindent
{Mathematics Subject Classification (2020): {Primary: 35J15, 35J25, Secondary: 35B45, 65N15}}\\

\noindent 
{Keywords: existence, regularity, uniqueness, elliptic equation, boundary value problems, singular coefficients
}

\section{Introduction}
In this paper, we investigate the existence, uniqueness, and quantitative estimates for weak solutions to the elliptic equation
\begin{equation} \label{maineqoursend}
\left\{
\begin{aligned}
-\operatorname{div}(\gamma \nabla u) + \langle \nabla \phi+ \mathbf{H}, \nabla u\rangle + (c+\alpha)\,u \;&=\; f
&\quad& \text{in } U,\\
u \;&=\; 0
&\quad& \text{on } \partial U.
\end{aligned}
\right.
\end{equation}
We work under the following main assumption: \\ \\
\noindent{\bf (Hy)}\; $\gamma \in [1,\infty)$ is a constant, $U \subset \mathbb{R}^d$ ($d \ge 2$) is a bounded open set, and $B_r(x_0)$ is an open ball satisfying $\overline{U} \subset B_r(x_0)$.
The potential function $\phi$ belongs to $H^{1,2}(U)\cap L^\infty(U)$, the vector field $\mathbf{H}$ belongs to $L^{p}(U, \mathbb{R}^d)$ for some $p\in(d,\infty)$, and the zero-order coefficient $c$ belongs to $L^1(U)$ with $c \geq 0$. \\  \\
The assumption $\phi\in H^{1,2}(U)\cap L^\infty(U)$ is essential for the exponential transformation introduced below. Throughout this paper, $\sup_U\phi$ and $\inf_U\phi$ denote the essential supremum and essential infimum of $\phi$ on $U$, respectively. \\ \\
By a weak solution to \eqref{maineqoursend}, we mean a function $u\in H^{1,2}_0(U)$ with $cu\in L^1(U)$ satisfying
$$
\int_U \left(
\langle \gamma\nabla u,\nabla\varphi\rangle
+\langle \nabla\phi+\mathbf H,\nabla u\rangle\varphi
+(c+\alpha)u\varphi
\right)\,dx=
\int_U f\varphi\,dx
$$
for every $\varphi\in C_0^\infty(U)$. The main contribution of this paper is not only to establish the
well-posedness of weak solutions under assumption {\bf (Hy)} and suitable integrability assumptions on the source term $f$, but also to derive explicit quantitative $L^2$ estimates of the form
\begin{equation} \label{quantitiestim3}
\|u\|_{L^2(U)} \leq C \|f\|_{L^2(U)}.
\end{equation}
A notable feature of this estimate is that an explicit expression for $C$ is obtained in terms of $K_1$ and the relevant parameters, and $C$ is independent of both $\|\nabla\phi\|_{L^2(U)}$ and $\|c\|_{L^1(U)}$. In particular, the dependence on the potential $\phi$ enters only through its oscillation (see Theorem \ref{thm:1.1}).
Consequently, the resulting $L^2$ estimate remains robust even when the $L^2$-norm of the low-regularity gradient $\nabla\phi$ is large, as long as the oscillation of $\phi$ is controlled. This feature is particularly useful for a posteriori error analysis, since the corresponding error constant depends on the oscillation of $\phi$ rather than on $\|\nabla\phi\|_{L^2(U)}$.
This robustness is particularly relevant to a posteriori error analysis for physics-informed neural networks (PINNs). Indeed, under the assumptions of Corollary \ref{cor:aposteriori}, the error $u-u_\theta$ satisfies the corresponding homogeneous Dirichlet problem with source term $L_\theta$, where $L_\theta$ is the residual defined in that corollary. Consequently, estimate \eqref{quantitiestim3} yields the computable a posteriori error bound
$$
\|u-u_\theta\|_{L^2(U)} \leq C\|L_\theta\|_{L^2(U)}.
$$ 
Such stability estimates provide a direct analytical mechanism for converting the residual loss of a PINN into a rigorous error bound. Related applications of quantitative stability and residual estimates to PINN error analysis can be found in \cite{YL24,L25,MM23}, while more general a posteriori error frameworks for elliptic problems and linear PDEs have been developed in \cite{BCP22,ZMM25}.\\
Our goal is to establish a well-posedness result for \eqref{maineqoursend} that remains robust in this non-symmetric setting and yields quantitative estimates with explicit dependence on the coefficients. Specifically, we quantify how the oscillation of the potential, $\sup_U \phi-\inf_U \phi$, affects the constants in the quantitative estimates through the variation of the effective weight.  A key novelty of the present work is that the drift component $\nabla \phi$ is only assumed to have $L^2(U,\mathbb{R}^d)$ regularity. Moreover, the zero-order coefficient is required to belong merely to $L^1(U)$ in the main well-posedness result, a range that lies beyond the standard assumptions in classical elliptic theories and related works on equations with lower-order terms (see, for example, \cite{S65,T73,GT01,Sha06,GGM13}). 
In Theorem \ref{stamconnec}, there is a substantial trade-off between the integrability imposed on the general vector field $\mathbf H$ and that imposed on the zero-order coefficient $c$. For instance, for $d\ge3$, Theorem \ref{stamconnec} requires $\mathbf{H}\in L^d(U,\mathbb{R}^d)$ and $c\in L^{\frac d2}(U)$ for well-posedness of \eqref{maineqoursend}. In contrast, our main result, Theorem \ref{thm:1.1}, establishes well-posedness under the assumption {\bf (Hy)}, namely under $\mathbf{H} \in L^p(U, \mathbb{R}^d)$ with $p \in (d, \infty)$, while assuming only $c \in L^1(U)$.\\
Unlike Theorem \ref{stamconnec}, Theorem \ref{thm:1.1} provides not only existence and uniqueness, but also a comprehensive quantitative analysis of the weak solution $u$. We now explain more explicitly how the present work differs from and builds on the closely related results \cite{L25jm}, \cite{L26c}, and \cite{L26y}.
In \cite{L25jm}, a well-posedness framework was developed for linear elliptic equations with general drift coefficients $\mathbf{H} \in L^p(U, \mathbb{R}^d)$ and $L^1$-zero-order terms. Related results for elliptic equations with singular drift terms under divergence-type conditions were obtained in \cite{L24,L25ej}. The present paper differs from \cite{L25jm} in that the drift additionally contains the gradient component $\nabla\phi$ with only $L^2(U,\mathbb{R}^d)$ regularity. The exponential transformation introduced in Theorem \ref{divfreeliktrans} absorbs this low-regularity gradient term into the diffusion coefficient and thereby makes it possible to apply the framework of \cite{L25jm} to the present equation.\\
In \cite{L26c}, quantitative $L^2$ estimates were established for linear elliptic equations via a divergence-free transformation. In the present paper, these quantitative techniques are applied to the transformed weighted equation and the resulting estimates are then transferred back to the original equation containing the additional $L^2$-gradient drift $\nabla\phi$. In particular, a notable feature of the resulting estimates is that their constants depend on the potential $\phi$ through its oscillation $\sup_U\phi-\inf_U\phi$, rather than through $\|\nabla\phi\|_{L^2(U)}$, and are independent of $\|c\|_{L^1(U)}$.\\
Finally, \cite{L26y} provides an interpolation result for well-posedness under different integrability assumptions on the source term and the zero-order coefficient. By combining this interpolation result with the exponential transformation of the present paper, Theorem \ref{thm:interpolation} establishes the corresponding trade-off between the integrability of $f$ and that of $c$ even in the presence of the additional $L^2$-gradient drift $\nabla\phi$. Moreover, using \cite[Theorem 1.1]{L26c}, we show that the estimates \eqref{quantitiestim-1} and \eqref{quantitiestim0} remain valid under these interpolated assumptions, while the $L^2$-based estimates \eqref{quantitiestim1}, \eqref{quantitiestim2}, and \eqref{quantitiestim3} hold under the additional assumption $f\in L^2(U)$.
\\
Central to our approach is the transformation introduced in Theorem \ref{divfreeliktrans}, which converts \eqref{maineqoursend} into the equivalent equation \eqref{maineqoutrans}. This transformation is motivated by the treatment of distorted Brownian motion in \cite{F81}, where drift terms with merely $L^2$ regularity are considered. It absorbs the $L^2$-singular drift $\nabla \phi$ into the diffusion coefficients and hence makes it possible to apply the analytical framework established in \cite{L25jm}. The appearance of the $L^2$-gradient perturbation is natural from the viewpoint of weighted energy forms and distorted Brownian motion. Low-regularity drift terms of this type can often be interpreted through changes of reference measure, weighted divergence structures, and associated Dirichlet form methods; see \cite{F81,FOT11,MR92}. The use of positive densities and divergence-type transformations is also closely related to the analytic study of invariant measures and stationary Fokker--Planck equations for singular diffusions; see, for example, \cite{BKR01} and \cite{BRS12}.
\\
The results show that the well-posedness theory for \eqref{maineqoursend} remains valid even in the presence of a drift with a low-regularity component in $L^2(U,\mathbb{R}^d)$ and a merely integrable zero-order term, while still yielding explicit quantitative estimates. They may also be viewed as analytic foundations for elliptic problems arising from singular diffusion models and stationary Fokker--Planck type equations. In particular, the quantitative estimates obtained here are relevant to the study of non-symmetric elliptic operators with low-regularity lower-order terms, which appear naturally in stochastic dynamics and applied mathematical models. \\
Our proofs are based on the exponential transformation introduced in Theorem \ref{divfreeliktrans}, which absorbs the $L^2$-gradient term $\nabla\phi$ into the diffusion coefficient. This allows us to apply the well-posedness framework of \cite{L25jm} and the quantitative estimates of \cite{L26c} to the transformed equation and then transfer the resulting conclusions back to the original problem. The same transformation, combined with the interpolation result of \cite[Theorem 5.1]{L26y}, also yields the trade-off between the integrability of $f$ and that of $c$ stated in Theorem \ref{thm:interpolation}.

\section{Main results}
\begin{thm}\label{divfreeliktrans}
Let $U$ be a bounded open subset of $\mathbb{R}^d$ with $d\ge 2$, and let $\gamma \in [1,\infty)$ be a constant.
Let $u\in H^{1,2}(U)$, $\phi\in H^{1,2}(U)\cap L^\infty(U)$, $\mathbf{H}\in L^2(U, \mathbb{R}^d)$, $c\in L^1(U)$ with $cu\in L^1(U)$ and $f \in L^1(U)$.
Fix a constant $k>0$ and set
\[
\eta := k e^{-\frac{\phi}{\gamma}}\qquad\text{in }U.
\]
Then the following are equivalent:
\begin{itemize}
\item[(i)]
\[
\int_U \Big(\langle \gamma \nabla u,\nabla \varphi\rangle
      + \langle \nabla \phi+\mathbf{H}, \nabla u\rangle\,\varphi
      + c\,u\,\varphi \Big)\,dx
= \int_U f\,\varphi\,dx,
\qquad \forall\,\varphi \in C_0^\infty(U).
\]
\item[(ii)]
\[
\int_U \Big(\langle \gamma \eta \nabla u, \nabla \psi\rangle
      + \langle \eta \mathbf{H}, \nabla u\rangle\,\psi
      + \eta c\,u\,\psi \Big)\,dx
= \int_U \eta f\,\psi\,dx,
\qquad \forall\,\psi \in C_0^\infty(U).
\]
\end{itemize}
\end{thm}
\begin{proof}
(i) $\Rightarrow$ (ii): Assume (i) holds. By \cite[Proposition A.8]{L25jm}, (i) holds for all $\varphi \in H^{1,2}_0(U) \cap L^{\infty}(U)$.
Let $\psi \in C_0^{\infty}(U)$ be arbitrary. Then, $\eta \psi \in H^{1,2}_0(U) \cap L^{\infty}(U)$.  Thus,
\[
\int_U \bigl\langle \gamma \nabla u, \nabla (\eta \psi) \bigr \rangle + \langle \nabla \phi+\mathbf{H}, \nabla u \rangle\,\eta \psi      + c\,u\,\eta \psi\, dx
= \int_U f\,\eta \psi \, dx.
\]
Observe that
\begin{align*}
&\int_U \bigl\langle \gamma \nabla u, \nabla (\eta \psi) \bigr \rangle + \langle \nabla \phi+ \mathbf{H}, \nabla u \rangle\,\eta \psi      + c\,u\,\eta \psi\, dx \\
&=\int_U \bigl\langle \gamma \eta \nabla u, \nabla \psi \bigr\rangle + \bigl\langle \gamma \nabla u, \nabla \eta \bigr\rangle \psi  + \langle \eta \nabla\phi, \nabla u \rangle \psi  + \langle \eta \mathbf{H}, \nabla u\rangle\,\psi      +\eta  c\,u\,\psi \, dx \\
&=\int_U \bigl\langle \gamma \eta \nabla u, \nabla \psi \bigr\rangle + \langle \mathbf{H}, \nabla u \rangle\,\eta \psi      + c\,u\,\eta \psi\, dx,
\end{align*}
and hence (ii) follows.\\ \\
(ii) $\Rightarrow$ (i):  Assume (ii) holds. By \cite[Proposition A.8]{L25jm}, (ii) holds for all $\psi \in H^{1,2}_0(U) \cap L^{\infty}(U)$. Let $\varphi\in C_0^{\infty}(U)$ be arbitrary. Then, $ \frac{\varphi}{\eta} \in H^{1,2}_0(U) \cap L^{\infty}(U)$. Thus,
$$
\int_U \bigl\langle \gamma \eta \nabla u, \nabla \left(\frac{\varphi}{\eta} \right) \bigr\rangle
      + \langle  \mathbf{H}, \nabla u\rangle\,\varphi  + c\,u\,\varphi \, dx
= \int_U f\,\varphi \, dx.
$$
Note that
\begin{align*}
&\int_U \bigl\langle \gamma \eta \nabla u, \nabla \left(\frac{\varphi}{\eta} \right) \bigr\rangle
      + \langle  \mathbf{H}, \nabla u\rangle\,\varphi  + c\,u\,\varphi \, dx\\
&= \int_U \bigl\langle \gamma \nabla u, \nabla \varphi \bigr\rangle
      + \langle \frac{-\gamma}{\eta}\nabla \eta+\mathbf{H}, \nabla u\rangle\,\varphi      + c\,u\,\varphi \, dx \\
&= \int_U \bigl\langle \gamma \nabla u, \nabla \varphi \bigr\rangle
      + \langle \nabla \phi+\mathbf{H}, \nabla u\rangle\,\varphi      + c\,u\,\varphi \, dx,
\end{align*}
and hence (i) follows.
\end{proof}

\begin{thm} \label{stamconnec}
Assume that $U\subset\mathbb{R}^d$ is a bounded open set with $d\ge 3$ and $\gamma \in [1, \infty)$, $\alpha \in [0, \infty)$ are constants. Let $\phi \in L^{\infty}(U) \cap H^{1,2}(U)$, $\mathbf{H}\in L^{d}(U, \mathbb{R}^d)$, $c\in L^{\frac{d}{2}}(U)$ with $c\ge 0$ a.e. in $U$ and $f \in L^{\frac{2d}{d+2}}(U)$. Then there exists a unique weak solution $u$ to \eqref{maineqoursend}.
\end{thm}

\begin{proof}
Set $\eta:=\frac{1}{\exp(-\sup_U \frac{\phi}{\gamma})} \exp(-\frac{\phi}{\gamma})$. Then, by the classical existence result from \cite{S65} (see \cite[Theorem 3.4]{L26y}), there exists a weak solution $u$ to
\begin{equation} \label{mainweighte}
\left\{
\begin{aligned}
-{\rm div}(\gamma \eta \nabla u) + \langle \eta \mathbf{H}, \nabla u\rangle + \eta  (c+\alpha)\,u \;&=\; \eta f
&\quad& \text{in } U,\\
u \;&=\; 0
&\quad& \text{on } \partial U.
\end{aligned}
\right.
\end{equation}
Then, by Theorem \ref{divfreeliktrans}, $u$ is a weak solution to \eqref{maineqoursend}. To show the uniqueness, let $v$ be a weak solution to \eqref{maineqoursend}. By Theorem \ref{divfreeliktrans}, $v$ is a weak solution to \eqref{mainweighte}. By the uniqueness result in \cite[Theorem 3.4]{L26y}, $u=v$ in $H^{1,2}_0(U)$, as desired.
\end{proof}

\noindent
Although the above theorem yields existence and uniqueness results under highly general drift coefficients, the associated quantitative estimates remain somewhat restricted (see Introduction in \cite{L26c}). To overcome this, we can derive explicit quantitative estimates by strengthening the integrability condition on the drift from $\mathbf{H} \in L^d(U, \mathbb{R}^d)$ to $\mathbf{H} \in L^p(U, \mathbb{R}^d)$ with $p \in (d, \infty)$, while significantly relaxing the assumption on the zero-order coefficient $c \in L^{\frac{d}{2}}(U)$.
\begin{lem} \label{existinvari}
Assume that {\bf (Hy)} holds. Set $\eta:=\frac{1}{\exp(-\sup_{U} \frac{\phi}{\gamma})} \exp(-\frac{\phi}{\gamma})$. Extend $\eta \equiv 1$ on $\mathbb{R}^d \setminus U$ and  $\mathbf H$ to $\mathbb R^d$ by setting $\mathbf H=0$ on $\mathbb R^d\setminus U$. Then, the following hold:
\begin{itemize}
\item[(i)] Let $x_1 \in U$. Then, there exists $\rho \in H^{1,2}(B_{4r}(x_0)) \cap C(B_{4r}(x_0))$
with $\rho(x) > 0$ for all $x \in B_{4r}(x_0)$ and $\rho(x_1)=1$ such that
\begin{equation*}
\int_{B_{4r}(x_0)} \bigl\langle \gamma \eta \nabla \rho + \rho \eta \mathbf{H}, \nabla \varphi \bigr\rangle \, dx = 0,
\qquad \text{for all } \varphi \in C_0^{\infty}(B_{4r}(x_0)).
\end{equation*}

\item[(ii)] Let $\rho$ be as in (i). Then, there exists a constant $K_1 \ge 1$ that depends only on $d$, $\frac{\sup_{U} \phi- \inf_{U} \phi}{\gamma}$, $r$, $p$, $\|\mathbf{H}\|_{L^p(U)}$ such that
\[
\max_{\overline{B}_{3r}(x_0)} \rho \le K_1 \min_{\overline{B}_{3r}(x_0)} \rho,
\]
and hence
\[
1 \le \max_{\overline{U}} \rho \le K_1 \min_{\overline{U}} \rho \le K_1.
\]
\end{itemize}
\end{lem}

\begin{proof}
Let us check the lower and upper bounds for $\eta$. Observe that
$$
1 \leq \eta \leq  \exp \left(\frac{\sup_{U} \phi - \inf_{U} \phi}{\gamma}\right) \quad \text{ in $\mathbb{R}^d$}.
$$
Thus, (i) and (ii) follow from \cite[Theorem 2.5]{L25jm}, after dividing the above divergence relation by $\gamma$ and applying the theorem with the diffusion coefficient $\eta$ and the drift $\eta\mathbf H/\gamma$. Indeed, the ellipticity ratio of $\eta$ is controlled by
$$
\exp\left(\frac{\sup_U\phi-\inf_U\phi}{\gamma}\right),
$$
and, since $\gamma\geq1$,
$$
\left\|\frac{\eta\mathbf H}{\gamma}\right\|_{L^p(U)}
\leq
\exp\left(\frac{\sup_U\phi-\inf_U\phi}{\gamma}\right)
\|\mathbf H\|_{L^p(U)}.
$$
Therefore, the Harnack constant in \cite[Theorem 2.5]{L25jm} can be chosen as a constant $K_1\geq1$ depending only on $d$, $\frac{\sup_U\phi-\inf_U\phi}{\gamma}$, $r$, $p$, and $\|\mathbf H\|_{L^p(U)}$, as stated in (ii).
\end{proof}

\begin{thm}\label{thm:1.1}
Suppose that {\bf (Hy)} holds. Define $\hat d$ by setting $\hat d=d$ when $d\ge3$, while for $d=2$ we choose an arbitrary $\hat d\in(2,\infty)$.
Let $c\in L^{1}(U)$ with $c\ge0$ and $\alpha \in [0,\infty)$ be a constant. 
Assume further that one of the following conditions is satisfied:
\begin{itemize}
\item[(a)] $d\ge2$ and
$f\in L^{q}(U)$ for some $q>\frac d2$.
\item[(b)] $d \geq 3$ and 
$f\in L^{\frac{2d}{d+2}}(U)$ and $c\in L^{\frac{2d}{d+2}}(U)$.
\end{itemize}
Then problem \eqref{maineqoursend} admits a unique weak solution $u\in H^{1,2}_0(U)$.
In the case of (a), we additionally have $u\in L^\infty(U)$ and
\[
\|u\|_{L^\infty(U)} \le C_2\,\|f\|_{L^{q}(U)},
\]
where $C_2>0$ depends on $d,q,p,|U|,\gamma,\alpha,\|\mathbf H\|_{L^p(U)}$, and $\sup_{U} \phi - \inf_U \phi$.
Furthermore, whenever either {(a)} or {(b)} holds, the solution $u$ enjoys the following estimates:
\begin{itemize}
\item[(i)] If $f\in L^{\frac{2\hat d}{\hat d+2}}(U)$, then
\begin{equation} \label{quantitiestim-1}
\|\nabla u\|_{L^{2}(U)}
\le \frac{K_{1}\hat k}{\gamma}\,
\|f\|_{L^{\frac{2\hat d}{\hat d+2}}(U)}
\end{equation}
and
\begin{equation} \label{quantitiestim0}
\|u\|_{L^{2}(U)}
\le
\left(\frac{d^{2}\gamma}{8(d-1)^{2}|U|^{\frac{2}{d}}}+\alpha\right)^{-\frac12}
\left(\frac{K_{1}^{2}\hat k^{2}}{2\gamma}\right)^{\frac12}
\|f\|_{L^{\frac{2\hat d}{\hat d+2}}(U)},
\end{equation}
where $K_{1}\ge 1$ is a constant as in Lemma \ref{existinvari} and
\[
\hat k:=\frac{\hat d}{\hat d-2}\,|U|^{\frac12-\frac1{\hat d}} \quad \text{if } d=2
\qquad\text{and}\qquad
\hat k:=\frac{2(d-1)}{d-2} \quad \text{if } d\ge 3.
\]

\item[(ii)] If $f\in L^{2}(U)$, then
\begin{equation}  \label{quantitiestim1}
\|u\|_{L^{2}(U)}
\le
K_{1}\left(\frac{d^{2}\gamma}{4(d-1)^{2}|U|^{\frac{2}{d}}}+\alpha\right)^{-1}
\|f\|_{L^{2}(U)}
\end{equation}
and
\begin{equation} \label{quantitiestim2}
\|u\|_{L^{2}(U)}
\le
K_{1}^{\frac12}\left(\frac{d^{2}\gamma}{4K_{1}(d-1)^{2}|U|^{\frac{2}{d}}}+\alpha\right)^{-1}
\|f\|_{L^{2}(U)}.
\end{equation}

\item[(iii)] Let
\begin{align*}
&C:=\min\Biggl(
\left(\frac{d^{2}\gamma}{8(d-1)^{2}|U|^{\frac{2}{d}}}+\alpha\right)^{-\frac12}
\left(\frac{K_{1}^{2}\hat k^{2}}{2\gamma}\right)^{\frac12}|U|^{\frac1{\hat d}},
\;
K_{1}\left(\frac{d^{2}\gamma}{4(d-1)^{2}|U|^{\frac{2}{d}}}+\alpha\right)^{-1}, \\
& \qquad \qquad \qquad \qquad K_{1}^{\frac12}\left(\frac{d^{2}\gamma}{4K_{1}(d-1)^{2}|U|^{\frac{2}{d}}}+\alpha\right)^{-1}
\Biggr),
\end{align*}
where $\hat k$ is defined as in (i). If $f\in L^{2}(U)$, then \eqref{quantitiestim3} holds.
\end{itemize}
\end{thm}
\begin{proof}
Set
$$
\eta:=\frac{1}{\exp(-\sup_U \frac{\phi}{\gamma})} \exp\left(-\frac{\phi}{\gamma}\right).
$$
Extend $\eta \equiv 1$ on $\mathbb{R}^d\setminus U$. Then,
$$
1\leq \eta \leq \exp\left(\frac{\sup_U\phi-\inf_U\phi}{\gamma}\right) \quad \text{in }\mathbb{R}^d.
$$
In particular, the transformed diffusion coefficient $\gamma\eta$ is bounded and uniformly elliptic. Moreover, since $\eta$ is bounded above and below by positive constants, we have
$$
\eta\mathbf H\in L^p(U,\mathbb{R}^d), \qquad \eta c\in L^1(U),
$$
with $\eta c\geq0$, and multiplication by $\eta$ preserves all the integrability assumptions imposed below on $f$ and $c$. We first consider case \textnormal{(a)}. Since $f\in L^q(U)$ for some $q>\frac d2$, we also have
$$
\eta f\in L^q(U).
$$
Thus, the coefficients and the source term in the transformed equation satisfy the assumptions of \cite[Theorem 1.1]{L26c}. Hence, there exists a unique weak solution $u\in H^{1,2}_0(U)$ of
\begin{equation}\label{maineqoutrans}
\left\{
\begin{aligned}
-\operatorname{div}(\gamma\eta\nabla u)+\langle\eta\mathbf H,\nabla u\rangle+\eta(c+\alpha)u
&=\eta f
&&\quad \text{in }U,\\
u&=0
&&\quad \text{on }\partial U.
\end{aligned}
\right.
\end{equation}
Moreover, $u\in L^\infty(U)$. Since $\eta$ is bounded and its upper bound depends only on $\frac{\sup_U\phi-\inf_U\phi}{\gamma}$, the corresponding estimate from \cite[Theorem 1.1]{L26c} yields
$$
\|u\|_{L^\infty(U)} \leq C_2\|f\|_{L^q(U)},
$$
where $C_2>0$ depends on $d,q,p,|U|,\gamma,\alpha,\|\mathbf H\|_{L^p(U)}$, and $\sup_U\phi-\inf_U\phi$. By Theorem \ref{divfreeliktrans}, every weak solution of \eqref{maineqoutrans} is a weak solution of \eqref{maineqoursend}. Therefore, the above function $u$ is a weak solution of
\eqref{maineqoursend}. Conversely, if $v$ is any weak solution of \eqref{maineqoursend}, then Theorem \ref{divfreeliktrans} implies that
$v$ is a weak solution of \eqref{maineqoutrans}. The uniqueness statement in \cite[Theorem 1.1]{L26c} therefore yields $u=v$ in $H^{1,2}_0(U)$. This proves existence and uniqueness in case \textnormal{(a)}. Next, assume \textnormal{(b)}. In this case,
$f,c\in L^{\frac{2d}{d+2}}(U)$, and the boundedness of $\eta$ implies
$\eta f,\eta c\in L^{\frac{2d}{d+2}}(U)$.
Hence, the corresponding assumptions of \cite[Theorem 1.1]{L26c} are again satisfied for the transformed problem
\eqref{maineqoutrans}. We therefore obtain a unique weak solution $u\in H^{1,2}_0(U)$ of \eqref{maineqoutrans}. Applying Theorem \ref{divfreeliktrans} in both directions as above, we conclude that $u$ is the unique weak solution of \eqref{maineqoursend}. This proves the well-posedness assertion in case \textnormal{(b)}. It remains to verify the quantitative estimates. By Lemma \ref{existinvari}, the transformed coefficients $\gamma\eta$ and $\eta\mathbf H$ admit a positive function $\rho$ satisfying the
corresponding divergence relation, together with the Harnack-type bound
$$
\max_{\overline{B}_{3r}(x_0)}\rho \leq K_1 \min_{\overline{B}_{3r}(x_0)}\rho,
$$
where $K_1\geq1$ depends only on $d$, $\frac{\sup_U\phi-\inf_U\phi}{\gamma}$, $r$, $p$, and
$\|\mathbf H\|_{L^p(U)}$. Thus, the quantitative estimates in \cite[Theorem 1.1]{L26c}, applied to the transformed problem
\eqref{maineqoutrans}, give \eqref{quantitiestim-1} and \eqref{quantitiestim0} whenever
$f\in L^{\frac{2\hat d}{\hat d+2}}(U)$, and give \eqref{quantitiestim1} and \eqref{quantitiestim2} whenever
$f\in L^2(U)$. Taking the minimum of the corresponding three constants then gives \eqref{quantitiestim3}. This completes the proof.
\end{proof}

\begin{cor}\label{cor:aposteriori}
Suppose that the assumptions of Theorem \ref{thm:1.1} hold, and let $u$ be the unique weak solution to \eqref{maineqoursend}. Let
$$
u_\theta\in C^\infty(\overline{U})\cap H^{1,2}_0(U),
$$
and define its residual by
$$
L_\theta:=f+\operatorname{div}(\gamma\nabla u_\theta)-\langle\nabla\phi+\mathbf H,\nabla u_\theta\rangle -(c+\alpha)u_\theta.
$$
Assume that $L_\theta\in L^2(U)$ and that one of the conditions {\rm (a)} or {\rm (b)} in Theorem \ref{thm:1.1} is satisfied with $f$ replaced by $L_\theta$. Then
$$
\|u-u_\theta\|_{L^2(U)}
\leq
C\|L_\theta\|_{L^2(U)},
$$
where $C$ is the constant defined in Theorem \ref{thm:1.1}{\rm (iii)}.

\end{cor}

\begin{proof}

Set
$$
e_\theta:=u-u_\theta.
$$
Since $u\in H^{1,2}_0(U)$ and
$u_\theta\in C^\infty(\overline U)\cap H^{1,2}_0(U)$, we have
$e_\theta\in H^{1,2}_0(U)$. Moreover, $u_\theta$ is bounded on $\overline U$, and hence $cu_\theta\in L^1(U)$. Since $cu\in L^1(U)$, it follows that $ce_\theta\in L^1(U)$.

By the definition of $L_\theta$ and the weak formulation satisfied by $u$, $e_\theta$ is a weak solution to
\begin{equation}\label{maineqourserror}
\left\{
\begin{aligned}
-\operatorname{div}(\gamma\nabla e_\theta)
+\langle\nabla\phi+\mathbf H,\nabla e_\theta\rangle
+(c+\alpha)e_\theta
&=L_\theta
&&\quad \text{in }U,\\
e_\theta&=0
&&\quad \text{on }\partial U.
\end{aligned}
\right.
\end{equation}
Therefore, Theorem \ref{thm:1.1}{\rm (iii)}, applied to \eqref{maineqourserror}, yields
$$
\|e_\theta\|_{L^2(U)}
\leq
C\|L_\theta\|_{L^2(U)},
$$
as desired.
\end{proof}

\begin{rmk}

Corollary \ref{cor:aposteriori} provides a direct a posteriori error estimate for residual-based approximations. In particular, if $u_\theta$ is a physics-informed neural-network (PINN) approximation as in Corollary \ref{cor:aposteriori} and the residual loss is defined by
$$
\mathcal J(\theta):=\|L_\theta\|_{L^2(U)}^2,
$$
then
$$
\|u-u_\theta\|_{L^2(U)}
\leq
C\,\mathcal J(\theta)^{1/2}.
$$
Thus, the residual loss provides a computable a posteriori bound for the $L^2$-error of the PINN approximation.

\end{rmk}

\begin{thm}\label{thm:interpolation}
Suppose that {\bf (Hy)} holds with $d \geq 3$ and let $\hat{p} \in (d, \infty)$. 
For any $r \in [2, d]$, let us define the interpolation exponent $k$ as
\[
    k := \frac{r(\hat{p}-2)}{2(\hat{p}-r)} \in [1, \infty).
\]
Assume that the lower-order term $c$ and the source term $f$ satisfy
\[
    c \in L^{\frac{2dk}{2dk-d+2}}(U) \quad \text{with } c \ge 0,
    \quad \text{and} \quad
    f \in L^{\frac{rd}{d+r}}(U).
\]
Let $\alpha \in [0,\infty)$ be a constant. 
Then, there exists a unique weak solution $u \in H^{1,2}_0(U) \cap L^{\frac{2dk}{d-2}}(U)$ to the equation \eqref{maineqoursend} such that $u$ satisfies the interpolated regularity estimate
\begin{equation} \label{interpolineq}
    \|u\|_{L^{\frac{2dk}{d-2}}(U)} \leq C_1^{\frac{1}{k}} C_2^{1-\frac{1}{k}} \|f\|_{L^{\frac{rd}{d+r}}(U)},
\end{equation}
where $C_1=\frac{4K_{1}(d-1)^2}{\gamma(d-2)^2}$ and $C_2>0$ is a constant as in Theorem \ref{thm:1.1} depending only on $d, q:=\frac{\hat{p}d}{d+\hat{p}},  p, |U|, \gamma, \alpha, \|\mathbf{H}\|_{L^{p}(U)}$ and $\sup_U \phi-\inf_{U} \phi$.
Furthermore, since $r \ge 2$ implies the embedding $L^{\frac{rd}{d+r}}(U) \hookrightarrow L^{\frac{2d}{d+2}}(U)$, the solution satisfies \eqref{quantitiestim-1} and \eqref{quantitiestim0}. Moreover, if in addition $f\in L^2(U)$, then \eqref{quantitiestim1}, \eqref{quantitiestim2}, and \eqref{quantitiestim3} also hold.
\end{thm}

\begin{proof}
By \cite[Theorem 5.1]{L26y}, there exists a weak solution $u$ to \eqref{maineqoutrans} such that \eqref{interpolineq} holds.
By Theorem \ref{divfreeliktrans}, $u$ is also a weak solution to \eqref{maineqoursend}.
Uniqueness of weak solutions to \eqref{maineqoursend} follows from \cite[Theorem 1.1(ii)]{L25jm}.
Let $f_n := (f \wedge n)\vee(-n)$. 
Since $f_n \in L^\infty(U)$, Theorem \ref{thm:1.1} yields a unique weak solution $u_n \in H^{1,2}_0(U) \cap L^{\infty}(U)$ to \eqref{maineqoursend} with $f$ replaced by $f_n$. Moreover, estimate \textnormal{(i)} in Theorem \ref{thm:1.1} gives
\begin{equation}\label{estimappint2}
\|\nabla u_n\|_{L^{2}(U)}
\leq \frac{K_{1}\hat k}{\gamma}\, \|f_n\|_{L^{\frac{2 d}{d+2}}(U)}
\leq \frac{K_{1}\hat k}{\gamma}\, \|f\|_{L^{\frac{2 d}{d+2}}(U)},
\end{equation}
and
\begin{align}
\|u_n\|_{L^{2}(U)}
&\le \left(\frac{d^{2}\gamma}{8(d-1)^{2}|U|^{\frac{2}{d}}}+\alpha\right)^{-\frac12}
\left(\frac{K_{1}^{2}\hat k^{2}}{2\gamma}\right)^{\frac12}  
\|f_n\|_{L^{\frac{2 d}{d+2}}(U)}  \nonumber \\
&\le
\left(\frac{d^{2}\gamma}{8(d-1)^{2}|U|^{\frac{2}{d}}}+\alpha\right)^{-\frac12}
\left(\frac{K_{1}^{2}\hat k^{2}}{2\gamma}\right)^{\frac12}
\|f\|_{L^{\frac{2 d}{d+2}}(U)}, \label{estimappint3}
\end{align}
where $\hat k:=\frac{2(d-1)}{d-2}$. Moreover, by \cite[Theorem 5.1]{L26y},
\begin{align*}
    \|u_n\|_{L^{\frac{2dk}{d-2}}(U)} \leq C_1^{\frac{1}{k}} C_2^{1-\frac{1}{k}} \|f_n\|_{L^{\frac{rd}{d+r}}(U)} \leq C_1^{\frac{1}{k}} C_2^{1-\frac{1}{k}} \|f\|_{L^{\frac{rd}{d+r}}(U)}.
\end{align*}
Then, by \eqref{estimappint2} and the above estimate in $L^{\frac{2dk}{d-2}}(U)$, there exist
$v \in H^{1,2}_0(U) \cap L^{\frac{2dk}{d-2}}(U)$ and a subsequence of $(u_n)_{n\ge1}$, still denoted by $(u_n)_{n\ge1}$, such that
\[
u_n \rightharpoonup v \quad \text{weakly in } H^{1,2}_0(U) \text{ and in }  L^{\frac{2dk}{d-2}}(U) \quad \text{as } n\to\infty.
\]
Passing to the limit in the weak formulation is justified by the weak convergence
$u_n \rightharpoonup v$ in $H^{1,2}_0(U)$ and in $L^{\frac{2dk}{d-2}}(U)$, the fact that
$c\varphi \in L^{\frac{2dk}{2dk-d+2}}(U)$ for every $\varphi \in C_0^\infty(U)$,
and the strong convergence $f_n \to f$ in $L^{\frac{rd}{d+r}}(U)$.
Therefore, $v$ is a weak solution to \eqref{maineqoursend}.
By the uniqueness established above, we have $v=u$. Moreover, by the lower semicontinuity of the $L^2$-norm with respect to weak convergence,
\[
\|u\|_{L^2(U)}
\leq \liminf_{n \rightarrow \infty}\|u_n\|_{L^{2}(U)}
\le
\left(\frac{d^{2}\gamma}{8(d-1)^{2}|U|^{\frac{2}{d}}}+\alpha\right)^{-\frac12}
\left(\frac{K_{1}^{2}\hat k^{2}}{2\gamma}\right)^{\frac12}
\|f\|_{L^{\frac{2 d}{d+2}}(U)}.
\]
Next, assume that $f \in L^2(U)$. 
Then, since $\|f_n\|_{L^2(U)} \le \|f\|_{L^2(U)}$, the estimate \textnormal{(ii)} in Theorem \ref{thm:1.1} yields
\begin{equation*}
\|u_n\|_{L^{2}(U)}
\le
K_{1}\left(\frac{d^{2}\gamma}{4(d-1)^{2}|U|^{\frac{2}{d}}}+\alpha\right)^{-1}
\|f\|_{L^{2}(U)},
\end{equation*}
and
\begin{equation*}
\|u_n\|_{L^{2}(U)}
\le
K_{1}^{\frac12}\left(\frac{d^{2}\gamma}{4K_{1}(d-1)^{2}|U|^{\frac{2}{d}}}+\alpha\right)^{-1}
\|f\|_{L^{2}(U)}.
\end{equation*}
Invoking the weak lower semicontinuity of the $L^2$-norm, we obtain
$$
\|u\|_{L^{2}(U)} \le \liminf_{n \to \infty} \|u_n\|_{L^{2}(U)} \le K_{1}\left(\frac{d^{2}\gamma}{4(d-1)^{2}|U|^{\frac{2}{d}}}+\alpha\right)^{-1} \|f\|_{L^{2}(U)},
$$
and similarly,
$$\|u\|_{L^{2}(U)} \le \liminf_{n \to \infty} \|u_n\|_{L^{2}(U)} \le K_{1}^{\frac12}\left(\frac{d^{2}\gamma}{4K_{1}(d-1)^{2}|U|^{\frac{2}{d}}}+\alpha\right)^{-1} \|f\|_{L^{2}(U)}.
$$
This establishes the required estimates for $f \in L^2(U)$ and completes the proof.
\end{proof}

\centerline{}
\centerline{}
Haesung Lee\\
Department of Mathematics and Big Data Science,  \\
Kumoh National Institute of Technology, \\
Gumi, Gyeongsangbuk-do 39177, Republic of Korea, \\
E-mail: fthslt@kumoh.ac.kr, \; fthslt14@gmail.com

\begin{thebibliography}{XXX}

\bibitem{BCP22}
S. Berrone, C. Canuto, M. Pintore,
{\it Solving PDEs by variational physics-informed neural networks: an a posteriori error analysis},
Ann. Univ. Ferrara {\bf 68} (2022), 575--595.

\bibitem{BKR01}
V. I. Bogachev, N. V. Krylov, M. R\"ockner,
{\it On regularity of transition probabilities and invariant measures of singular diffusions under minimal conditions},
Comm. Partial Differential Equations {\bf 26} (2001), no.~11--12, 2037--2080.


\bibitem{BRS12}
V. I. Bogachev, M. R\"ockner, S. V. Shaposhnikov,
{\it On positive and probability solutions to the stationary Fokker--Planck--Kolmogorov equation},
Dokl. Math. {\bf 85} (2012), no.~3, 350--354.


\bibitem{F81}
M. Fukushima,
{\it On a stochastic calculus related to Dirichlet forms and distorted Brownian motions},
Phys. Rep. {\bf 77} (1981), no.~3, 255--262.

\bibitem{FOT11}
M. Fukushima, Y. Oshima, M. Takeda,
{\it Dirichlet forms and symmetric Markov processes},
2nd revised and extended ed., De Gruyter Studies in Mathematics, vol.~19,
De Gruyter, Berlin, 2011.

\bibitem{GGM13}
F. Giannetti, L. Greco, G. Moscariello,
{\it Linear elliptic equations with lower-order terms},
Differ. Integral Equ. {\bf 26} (2013), no.~5--6, 623--638.

\bibitem{GT01}
D. Gilbarg, N. S. Trudinger,
{\it Elliptic partial differential equations of second order},
Classics in Mathematics, Springer-Verlag, Berlin, 2001.

\bibitem{L24}
H. Lee,
{\it On the contraction properties for weak solutions to linear elliptic equations with $L^2$-drifts of negative divergence},
Proc. Amer. Math. Soc. {\bf 152} (2024), no.~5, 2051--2068.

\bibitem{L25jm}
H. Lee,
{\it Analysis of linear elliptic equations with general drifts and $L^1$-zero-order terms},
J. Math. Anal. Appl. {\bf 548} (2025), 129425.

\bibitem{L25}
H. Lee,
{\it On the stability of linear elliptic equations with $L^2$-drifts of negative divergence and singular zero-order terms},
Funct. Anal. Appl. {\bf 59} (2025), no.~4, 405--420.

\bibitem{L25ej}
H. Lee,
{\it Well-posedness of linear elliptic equations with $L^d$-drifts under divergence-type conditions},
Electron. Res. Arch. {\bf 33} (2025), no.~12, 7974--7998.

\bibitem{L26c}
H. Lee,
{\it Quantitative analysis for $L^2$-estimates in linear elliptic equations via divergence-free transformation},
J. Chungcheong Math. Soc. {\bf 39} (2026), no.~1, 33--45.

\bibitem{L26y}
H. Lee,
{\it Remarks on well-posedness for linear elliptic equations via divergence-free transformation},
East Asian Math. J. {\bf 42} (2026), no.~1, 35--50.

\bibitem{MR92}
Z. M. Ma, M. R\"ockner,
{\it Introduction to the theory of (non-symmetric) Dirichlet forms},
Universitext, Springer-Verlag, Berlin, 1992.


\bibitem{MM23}
S. Mishra, R. Molinaro,
{\it Estimates on the generalization error of physics-informed neural networks for approximating PDEs},
IMA J. Numer. Anal. {\bf 43} (2023), no.~1, 1--43.


\bibitem{Sha06}
S. V. Shaposhnikov,
{\it On Morrey's estimate for the Sobolev norms of solutions of elliptic equations},
Math. Notes {\bf 79} (2006), no.~3--4, 413--430.

\bibitem{S65}
G. Stampacchia,
{\it Le probl\`eme de Dirichlet pour les \'equations elliptiques du second ordre \`a coefficients discontinus},
Ann. Inst. Fourier (Grenoble) {\bf 15} (1965), no.~1, 189--258.

\bibitem{T73}
N. S. Trudinger,
{\it Linear elliptic operators with measurable coefficients},
Ann. Scuola Norm. Sup. Pisa Cl. Sci. (3) {\bf 27} (1973), 265--308.

\bibitem{YL24}
J. Yoo, H. Lee,
{\it Robust error estimates of PINN in one-dimensional boundary value problems for linear elliptic equations},
AIMS Math. {\bf 9} (2024), no.~10, 27000--27027.

\bibitem{ZMM25}
M. Zeinhofer, R. Masri, K.-A. Mardal,
{\it A unified framework for the error analysis of physics-informed neural networks},
IMA J. Numer. Anal. {\bf 45} (2025), no.~5, 2988--3025.




\end{thebibliography}
\end{document}